\documentclass[11pt]{article}

\usepackage[margin=1in]{geometry}
\usepackage{amsmath,amssymb,amsthm}
\usepackage{hyperref}
\usepackage{mathdots}
\usepackage{xcolor}

\theoremstyle{plain}
\newtheorem{theorem}{Theorem}[section]
\newtheorem{lemma}[theorem]{Lemma}

\newtheorem{proposition}[theorem]{Proposition}

\theoremstyle{definition}
\newtheorem{definition}[theorem]{Definition}
\newtheorem{example}[theorem]{Example}

\theoremstyle{remark}
\newtheorem{remark}[theorem]{Remark}

\newcommand{\C}{\mathbb C}

\newcommand{\f}{{\bf f}}
\newcommand{\h}{{\bf h}}
\newcommand{\g}{{\bf g}}
\newcommand{\cc}{{\bf c}}
\newcommand{\dd}{{\bf d}}
\newcommand{\Ab}{{\bf{\mathcal A}}}

\title{On the Coefficients of Hurwitz-Type Matrix Polynomials}

\author{
A.~E. Choque-Rivero\\[1mm]
\small Institute of Physics and Mathematics,\\
\small Michoacan University of Saint Nicholas of Hidalgo,\\
\small Ciudad Universitaria, Edificio C3-A,\\
\small Morelia 58030, Michoacan, Mexico
}

\date{}

\begin{document}

\maketitle

\begin{abstract}
Consider the matrix polynomial
$
\mathbf f_n(z)=I_qz^n+A_1z^{\,n-1}+\cdots+A_n,
$
where \(A_j\in\mathbb C^{q\times q}\). Write
$
\mathbf f_n(z)=
\mathbf h_n(z^2)+z\,\mathbf g_n(z^2).
$
A matrix polynomial \(\mathbf f_n\) is called a
\emph{Hurwitz-type matrix polynomial} if, for \(n=2m\),
\(\mathbf g_n(z)\mathbf h_n(z)^{-1}\), and for \(n=2m+1\),
\(\mathbf h_n(z)\bigl(z\mathbf g_n(z)\bigr)^{-1}\),
admit finite continued fraction expansions with positive definite
coefficients.

We derive explicit formulas for the coefficients of Hurwitz-type
matrix polynomials in terms of orthogonal matrix polynomials, Markov
parameters, and Schur complements. We introduce the associated block
Hurwitz matrix and establish determinant identities relating it to the
corresponding block Hankel matrices. Finally, we settle a conjecture
concerning the positivity of the determinants of the coefficient
matrices of Hurwitz-type matrix polynomials. We prove that this
property holds for degrees at most three, but fails in degree four by
means of an explicit counterexample.
\end{abstract}

\medskip
\noindent\textbf{Keywords:}
Hurwitz-type matrix polynomials;
Hurwitz matrix polynomials;
orthogonal matrix polynomials;
Schur complements;
Markov parameters.

\medskip
\noindent\textbf{2020 Mathematics Subject Classification:}
15A15, 33C45, 30E05, 47A56, 15A24.

\section{Introduction}\label{sec1}
The set \(\mathbb C^{q\times q}\) denotes the space of all complex
\(q\times q\) matrices.
We write \(0_q\) and \(I_q\) for the zero and identity
matrices, respectively, and omit the indices whenever the dimensions
are clear from the context.
Let \(n\) and \(q\) be natural numbers. 

A scalar polynomial is called a Hurwitz polynomial if all of its roots
lie in the left half of the complex plane
\cite{post,datta,furhmann,holtz,ab2018,jaro,jury}.
Every real Hurwitz polynomial necessarily has positive coefficients.
This  leads to the question of whether an analogous property
holds for matrix polynomials, namely, whether the determinants of the
coefficient matrices are positive.

To address this question, in the present work we investigate a class
of matrix polynomials, introduced in \cite{abH}, belonging to the
class of lambda--matrices \cite[Page~95]{macd}, \cite{dennis},
also called matrix polynomials \cite{lan01}, of the form
\begin{equation}
\mathbf f_n(z)
=
I_q z^n+A_1z^{n-1}+\cdots+A_n,
\label{hup01}
\end{equation}
where \(A_j\in\mathbb C^{q\times q}\) and \(z\) is a complex variable.

Following
\cite{heinig,lan01,lerer1,martins,resende,shieh},
we call the matrix polynomial \(\mathbf f_n\)
Hurwitz if its determinant
\(\det\mathbf f_n\) is a Hurwitz polynomial.

Since the leading coefficient of \(\mathbf f_n\) is \(I_q\), the
matrix polynomial~\eqref{hup01} has degree \(n\). Every matrix
polynomial of degree \(n\) admits the decomposition
\begin{equation}
\mathbf f_n(z)
=
\mathbf h_n(z^2)+z\,\mathbf g_n(z^2),
\label{hup02}
\end{equation}
where
\begin{align}
\h_n(z) &:= \begin{cases}
I_q z^m + A_{2} z^{m-1} + \ldots + A_{2m}, & n = 2m, \\[0.3em]
A_{1} z^m + A_{3} z^{m-1} + \ldots + A_{2m+1}, & n = 2m+1,
\end{cases} \label{hn} \\[0.5em]
\g_n(z) &:= \begin{cases}
A_{1} z^{m-1} + A_{3} z^{m-2} + \ldots + A_{2m-1}, & n = 2m, \\[0.3em]
I_q z^m + A_{2} z^{m-1} + \ldots + A_{2m}, & n = 2m+1.
\end{cases} \label{gn}
\end{align}
We now recall the definition of Hurwitz-type matrix polynomials,
which is formulated in terms of the decomposition~(2).
For \(A,B\in\mathbb C^{q\times q}\), with \(B\) invertible,
we adopt the notation
\[
\frac{A}{B}:=AB^{-1}.
\]
 \begin{definition}\label{defhur}
The $q\times q$ matrix polynomial $\f_n$ in (\ref{hup01}) is called
a Hurwitz-type matrix polynomial  (HTM) if there exist two sequences of
positive definite matrices, $(\cc_{k})_{k=0}^{m-1}$ and $(\dd_{k})_{k=0}^{m-1}$ 
(resp.  $(\cc_{k})_{k=0}^m$ and $(\dd_{k})_{k=0}^{m-1}$) such that, for $n=2m$,
\begin{align}
 \frac{\g_n (z)}{\h_n(z)}=
 &\cfrac{I_q}{z\cc_{0}+\cfrac{I_q}{\dd_{0}
 +\cfrac{I_q}{\ddots 
 +\dd_{m-2}+
 \cfrac{I_q}{
+z\cc_{m-1}+\dd_{m-1}^{-1}}}}}, \label{cc13aa}
\end{align}
for all $z\in \C$ with $\det \h_n(z)\neq0$ and, for $n=2m+1$,
\begin{align}
\frac{\h_n (z)}{z \g_n(z)}=
&\cfrac{I_q}{z\cc_{0}+\cfrac{I_q}{\dd_{0}+\cfrac{I_q}{\ddots \,
 +z\cc_{m-1}+\cfrac{I_q}{\dd_{m-1}
+z^{-1}\cc_{m}^{-1} }}}} \label{cc23aa}
 \end{align}
 for all $z\in {\mathbb C} \setminus\{0\}$ with $\det \g_n(z)\neq0$.
\end{definition}
The continued-fraction representation in Definition~\ref{defhur} is
motivated by Theorem~\cite[Chapter~XV, Theorem~15]{gant0}.

For scalar real polynomials of degree at most two, positivity of the
coefficients is both necessary and sufficient for the Hurwitz
property. This characterization, however, no longer holds for matrix
polynomials.

Indeed, positivity of the determinants of the coefficient matrices
alone does not imply the Hurwitz property. For instance,
\(\mathbf f_1(z)=I_2z-A_1\), where \(A_1=I_2\).
Then
$
\det A_1=1>0,
\, 
\det\mathbf f_1(z)=(z-1)^2,
$
which is not Hurwitz.

The notion of Hurwitz-type matrix (HTM) polynomials was introduced
in \cite{abH}. Among the main results of that work is a
representation of HTM polynomials in terms of orthogonal matrix
polynomials on \([0,\infty)\) and the corresponding second-kind
matrix polynomials.

Subsequently, the fact that every HTM polynomial is Hurwitz was
suggested and partially established using Bezoutians in
\cite{zhan1}.

Further properties of HTM polynomials were discussed in
\cite{abcomment}. In particular, the coprimeness of the polynomial
pairs arising in the Hurwitz decomposition and their connection with
the truncated matrix Stieltjes moment problem
\cite{dyuth,dyu01,dyu2009,fkm02,fkm03}
were analyzed.

Later, in \cite{ab2026}, explicit Bezoutian representations for
Hurwitz-type matrix (HTM) polynomials were obtained, together with a
complete proof that every HTM polynomial is Hurwitz. The paper also
introduced a degree-doubling transformation associating an arbitrary
matrix polynomial with an HTM polynomial, yielding new criteria for
Hurwitz stability. In the same work, it was conjectured
(\cite[Remark~13]{ab2026}) that every coefficient matrix of an HTM
polynomial has positive determinant.

More recently, these results were applied in \cite{choqueTAC} to the
robust Hurwitz stability of parameter-dependent matrix polynomials,
yielding explicit Bezoutian-based criteria for multivariable
state-feedback stabilization.

Motivated by the above conjecture, the main objective of the present
paper is to investigate this conjecture and to describe the coefficient
matrices explicitly in terms of the associated moment-theoretic
quantities.

An important application of HTM polynomials arises in the analysis of
Hurwitz stability for linear time-invariant systems of the form
\begin{equation}\label{sys2}
\dot y=\widehat A y, \qquad \widehat A=
\begin{pmatrix}
-A_1&-A_2&\cdots&-A_{n-1}&-A_n\\
I_q&0_q&\cdots&0_q&0_q\\
0_q&I_q&\cdots&0_q&0_q\\
\vdots&\vdots&\ddots&\vdots&\vdots\\
0_q&0_q&\cdots&I_q&0_q
\end{pmatrix},
\end{equation}
where \(A_j\in\mathbb C^{q\times q}\) and
\(y\in\mathbb C^{nq}\).

The matrix \(\widehat A\) is the block companion linearization of the
matrix polynomial \eqref{hup01}. Consequently,
\begin{equation}\label{charpoly}
\det\!\bigl(zI_{nq}-\widehat A\bigr)=\det\!\bigl(\mathbf f_n(z)\bigr).
\end{equation}
Hence, the spectrum of \(\widehat A\) coincides, counting algebraic
multiplicities, with the set of zeros of
\(\det\!\bigl(\mathbf f_n(z)\bigr)\); see
\cite[p.~14]{gohberg}.

Another application of HTM polynomials is the design of
multi-input state-feedback controllers for linear MIMO systems such that every trajectory starting in a
neighborhood of the origin reaches the origin in finite time.
The scalar case was investigated in
\cite{choque5,choque6}.

\noindent\textbf{Main contributions.}
We establish the following results.
\begin{itemize}
\item[(a)]
We derive explicit formulas for the coefficient matrices of HTM
polynomials in terms of orthogonal matrix polynomials,
second-kind matrix polynomials and Schur complements. 
We also show how to reconstruct the associated Markov
parameters from the polynomial coefficients.

\item[(b)]
We introduce a block Hurwitz matrix associated with an HTM
polynomial and establish explicit factorizations and determinant
identities. As an application, we settle the conjecture from
\cite[Remark~13]{ab2026} by proving positivity of the coefficient
determinants for HTM polynomials of degree at most three and by
constructing a counterexample in degree four.

\item[(c)]
We present several examples illustrating the theory,
including the computation of orthogonal matrix polynomials, Markov
parameters, block Hurwitz matrices, and HTM polynomials
constructed from a Stieltjes measure.
\end{itemize}

The paper is organized as follows. In Section~\ref{sec2}, we review
the basic results on Stieltjes orthogonal matrix polynomials,
second-kind matrix polynomials, and the corresponding auxiliary
matrices.

In Section~\ref{sec3}, we derive explicit relations between the
Markov parameters and the coefficients of HTM polynomials.

Section~\ref{sec4} contains explicit representations of the
coefficient matrices of HTM polynomials in terms of orthogonal matrix
polynomials, second-kind matrix polynomials, Schur complements, and
Hankel block matrices.

In Section~\ref{sec5}, we give several examples illustrating the
theory, including the construction of HTM polynomials from Stieltjes
measures.

In Section~\ref{sec6}, we study the determinants of the coefficient
matrices. We prove positivity for HTM polynomials of degree at most
three and give a counterexample for degree four.

Finally, in Section~\ref{sec7}, we introduce the associated block
Hurwitz matrix, establish its relation with the Markov parameters and
Hankel block matrices, and derive determinant identities for the
leading block Hurwitz determinants.

\section{Preliminaries} \label{sec2}
In this section, we recall from \cite{ablaa,abH} several auxiliary
matrices together with two families of orthogonal matrix polynomials
and their corresponding second-kind matrix polynomials.

\subsection{Markov parameters associated with HTM polynomials}
 The following result, reproduced from
\cite[Lemma~2.9]{abH}, shows that the rational matrix functions
$
\g_n/\h_n
$
and
$
\h_n/(z\g_n)
$
admit Laurent expansions at infinity whose coefficients are the
associated Markov parameters.

\begin{remark}\label{rem2.1}
 Let $\h_n$ and $\g_n$ be defined as in (\ref{hn}) and (\ref{gn}) and let $\f_n$ be as in (\ref{hup02}).
  Let $|z_0|\geq \max \{|\lambda|: \det \h_n(\lambda)=0\}$
  (resp.\ $|z_0|\geq \max \{|\lambda|: \det \g_n(\lambda)=0\}$). For $|z|>|z_0|$,
the power series expansion of $\g_n/\h_n$ (resp.\ $\h_n/\g_n$) in
negative powers of $z$,
\begin{align}
\frac{\g_n(z)}{\h_n(z)}&=\frac{s_0}{z}-\frac{s_1}{z^2}+\ldots+(-1)^{n}\frac{s_n}{z^{n+1}}+\ldots\,,
\ \
n=2m,\label{hup09}\\
\frac{\h_n(z)}{z\g_n(z)}&=\frac{s_0}{z}-\frac{s_1}{z^2}+\ldots+(-1)^{n}\frac{s_n}{z^{n+1}}+\ldots\,,
\ \ n=2m+1,
 \label{hup10}
 \end{align}
\end{remark}

The rational matrix functions
\[
\frac{\g_{2m}(-z)}{\h_{2m}(-z)}
\quad\text{and}\quad
\frac{\h_{2m+1}(-z)}{(-z)\g_{2m+1}(-z)}
\]
are the extremal solutions of the truncated Stieltjes matrix moment
 problem; see the proof of Theorem~6.1 in \cite{abH}.

The corresponding Markov parameters
\((s_j)_{j\ge0}\)
are precisely the Stieltjes moments of the measure
\(\sigma\), namely,
\begin{equation}\label{ss-j}
s_j=\int_{[0,\infty)}t^j\,\sigma(dt),\qquad j\ge0.
\end{equation}
where \(\sigma\) is a nonnegative Hermitian \(q\times q\) measure on
\([0,\infty)\).

Associated with the Markov parameters \((s_j)_{j\ge0}\)
are the block Hankel matrices \(H_{1,j}\) and \(H_{2,j}\), defined by
\begin{equation} \label{61}
H_{1,j}:=
\begin{pmatrix}
s_{0} & s_{1}& \ldots & s_{j}\\
s_{1} & s_{2}& \ldots & s_{j+1}\\
\vdots  & \vdots  & \vdots & \vdots\\
s_{j} & s_{j+1} & \ldots & s_{2j}
\end{pmatrix},
H_{2,j}:=\begin{pmatrix}
s_{1} & s_{2}& \ldots & s_{j+1}\\
s_{2} & s_{3}& \ldots & s_{j+2}\\
\vdots  & \vdots  & \vdots & \vdots\\
s_{j+1} & s_{j+2} & \ldots & s_{2j+1}
\end{pmatrix}.
\end{equation}

 \begin{definition}\label{posseq}
  For all $m\in\mathbb{N}\cup {0}$, a sequence
\((s_j)_{j=0}^{2m+1}\) of complex \(q\times q\) matrices is called a
\emph{Stieltjes positive sequence} if the corresponding block Hankel
matrices \(H_{1,m}\) and \(H_{2,m}\) (resp.\ \(H_{1,m}\) and
\(H_{2,m-1}\) for \(m\ge1\)) are positive definite.
\end{definition}
Throughout the paper, we assume that all moment sequences under
consideration are Stieltjes positive sequences.
 
Let
\begin{align}
Y_{1,j} &:= y_{[j,\,2j-1]}, \qquad Y_{2,j} := y_{[j+1,\,2j]}, \quad 1 \leq j \leq n,
\label{81}
\end{align}
where
\begin{equation}
y_{[j,k]}:=\left(s_j,\, s_{j+1},\, \ldots,\, s_k\right)^*,\qquad 0\le j\le k,\label{65}
\end{equation}
with the convention that $y_{[j,k]}:=0_q,\, j>k.$

Furthermore, let
\(\widehat H_{1,j}\)
(resp.\ \(\widehat H_{2,j}\))
denote the Schur complement of the block
\(H_{1,j-1}\)
in
\(H_{1,j}\)
(resp.\ of
\(H_{2,j-1}\)
in
\(H_{2,j}\));
see~\cite{ouellette} and
\cite[Equalities~(2.5) and~(2.6)]{ablaa}.
These matrices are given by
\begin{align}
\widehat H_{1,0}
&:=s_0,\quad  \widehat H_{1,j} :=s_{2j}-Y_{1,j}^* H_{1,j-1}^{-1} Y_{1,j},\quad j\ge1, \label{hh1j}
\\
\widehat H_{2,0}
&:=s_1,\quad \widehat H_{2,j}:=s_{2j+1}-Y_{2,j}^*H_{2,j-1}^{-1}Y_{2,j},\quad j\ge1.
\label{hh2j}
\end{align}
Since the moment sequence is Stieltjes positive, the block Hankel
matrices \(H_{1,j}\) and \(H_{2,j}\), together with their Schur
complements \(\widehat H_{1,j}\) and \(\widehat H_{2,j}\), are
Hermitian positive definite.

\subsection{Stieltjes orthogonal matrix polynomials} \label{sub1-2}
 Let
  $R_{j}:\mathbb{C}\to \mathbb{C}^{(j+1)q\times (j+1)q}$  be given by
 \begin{align}
\label{52}R_{j}(z)&:=(I_{(j+1)q}-zT_{j})^{-1},\qquad j\geq 0,
\end{align}
 with
\begin{equation}
T_{0}:=0_q, \ \  T_{j}:=\left(
\begin{array}{cc}
0_{q\times jq} & 0_q\\
I_{jq} & 0_{jq\times q}\\
                \end{array}
         \right),\qquad j\geq 1.\label{56}
 \end{equation}
 Note that for each $j \in \mathbb{N}_0$, the matrix-valued function $R_{j}$ admits the representation
\begin{equation}
R_{j}(z)=\left(
         \begin{array}{cccccc}
           I_q & 0_q & 0_q & \ldots & 0_q & 0_q \\
           zI_q &  I_q  & 0_q & \ldots & 0_q & 0_q\\
           z^2 I_q & z I_q & I_q & \ldots & 0_q & 0_q \\
 \vdots & \vdots & \vdots & \iddots & \vdots & \vdots \\
           z^jI_q & z^{j-1}I_q & z^{j-2}I_q &  \ldots& zI_q & I_q
         \end{array}
       \right).\label{o1.6}
\end{equation}
Let
\begin{align}
v_{0}
&:= I_q,
\qquad
v_{j}
:=
\begin{pmatrix}
I_q\\
0_{jq\times q}
\end{pmatrix},
\qquad
j\in\mathbb N.
\label{60a}
\end{align}

Furthermore, let \(y_{[j,k]}\) be defined by \eqref{65}. For
\(1\le j\le n\), define
\begin{align}
u_{1,0}
&:=0_q,
\qquad
u_{1,j}
:=
\begin{pmatrix}
0_q\\
-y_{[0,j-1]}
\end{pmatrix},
\qquad
u_{2,j}
:=
-y_{[0,j]}.
\label{66}
\end{align}
We now recall the notion of monic left orthogonal matrix
polynomials associated with Hankel positive definite moment
sequences (see \cite[Definition~3.2]{abmad}).
To this end, we first recall the standard notions of degree and
leading coefficient of a matrix polynomial.
For a complex $p \times q$ matrix polynomial $P(z) = \sum_{j=0}^\infty z^j A_j$  
with coefficients $A_j \in \mathbb{C}^{p\times q}$, we set  
$
Z^{[P]}_{n} := [A_0, A_1, \dotsc, A_n], \quad n \in \mathbb{N}_0,
$
and define  
$
\deg P := \sup\{ j \in \mathbb{N}_0 \mid A_j \neq 0_{p\times q} \},
$
taking $\deg P = -\infty$ if $P \equiv 0_{p\times q}$.  
When $\deg P = k \ge 0$, the coefficient $A_k$ is the \emph{leading coefficient} of $P$.

\begin{definition}\label{def3.2A}
Let \(\kappa\in\mathbb N_0\cup\{\infty\}\), and let
\((\widetilde s_j)_{j=0}^{2\kappa}\) be a sequence of complex
\(q\times q\) matrices. For each \(n\in\mathbb N_0\), define the
block Hankel matrix
\[
H_n^{(\widetilde s)}
:=
(\widetilde s_{j+k})_{j,k=0}^{n}.
\]

The sequence
\((\widetilde s_j)_{j=0}^{2\kappa}\)
is called \emph{Hankel positive definite} if
\(H_n^{(\widetilde s)}\) is Hermitian positive definite for every
\(n\le\kappa\).

A sequence \((P_k)_{k=0}^{\kappa}\) of complex
\(q\times q\) matrix polynomials is called a
\emph{monic left orthogonal system of matrix polynomials with respect
to \((\widetilde s_j)_{j=0}^{2\kappa}\)} if the following conditions
are fulfilled:
\begin{itemize}
\item[(I)]
\(\deg P_k=k\) for all \(k\in\mathbb Z_{0,\kappa}\).
\item[(II)]
The leading coefficient of \(P_k\) is \(I_q\).
\item[(III)]
\[
Z_n^{[P_j]}
H_n^{(\widetilde s)}
\bigl(Z_n^{[P_k]}\bigr)^*
=
0_{q\times q},
\]
for all \(j\neq k\), where
\(n=\max\{j,k\}\).
\end{itemize}
\end{definition}
In what follows, Definition~\ref{def3.2A} is applied to the two
moment sequences
$(\widetilde s_j)~=~(s_j)$,
and
$(\widetilde s_j)~=~(s_{j+1})$,
which give rise to the orthogonal matrix polynomial systems
\((P_{1,j})\) and \((P_{2,j})\), respectively.

\begin{definition}
\label{def2p}
  Let $(s_{j})_{j=0}^{2n}$ (resp.\ $(s_{j})_{j=0}^{2n+1}$) be a Stieltjes positive sequence. Furthermore, let 
$v_j$, $u_{1,j}$, $u_{2,j}$ and $R_j$ be as in 
(\ref{60a}),  (\ref{66})  and (\ref{o1.6}), respectively. Let
\begin{equation}
 P_{1,0}(z):=I_q, \ \ Q_{1,0}(z):=0_q, \ \ P_{2,0}(z):=I_q, \ \
Q_{2,0}(z):=s_0. \label{pq00}
\end{equation}
For $j\geq 1$,  set
\begin{align}
P_{1,j}(z):=&(-Y_{1,j}^* H_{1,j-1}^{-1},\,I_q)R_j(z)v_j, \label{p1j}\\
P_{2,j}(z):=&(-Y_{2,j}^*  H_{2,j-1}^{-1},\,I_q)R_j(z)v_j, \label{p2j}\\
Q_{1,j}(z):=&-(-Y_{1,j}^*
H_{1,j-1}^{-1},\,I_q)R_j(z)u_{1,j},\label{q1j}\\
Q_{2,j}(z):=&-(-Y_{2,j}^*  H_{2,j-1}^{-1},\,I_q)R_j(z)u_{2,j}. \label{qj2}
\end{align}
The matrix polynomials $Q_{1,j}$ and $Q_{2,j}$
 are called second-kind polynomials with respect to $P_{1,j}$ and  $P_{2,j}$,
respectively.
\end{definition}
These matrix polynomials appear to have been introduced for the first time in \cite{dyuth}.

\vskip2mm

The following proposition, reproduced from
\cite[Proposition~4.9(b)]{ablaa},
establishes
  relationships among the matrix polynomials \( P_{1,j} \), \( Q_{1,j} \), \( P_{2,j} \), \( Q_{2,j} \), and the corresponding Schur complements.
\begin{proposition} \label{prop4.3}
 Let $P_{k,j}$, $Q_{k,j}$ be as in Definition
\ref{def2p}, and let $\widehat H_{k,j},\, k=1,2$ be  as in
\eqref{hh1j}, \eqref{hh2j}.
 Then the following identities hold:
\begin{align}
 Q_{2,j}^*(\bar z) Q_{2,j}^{*^{-1}}(0)=&I_q+z \sum_{l=0}^j Q_{1,l}^*(\bar z)
  \widehat H_{1,l}^{-1}P_{1,l}(0), \quad j\geq0, \label{B11j}
  \\
Q_{1,j}^*(\bar z) P_{1,j}^{*^{-1}}(0)=&-\sum_{l=0}^{j-1}
Q_{2,l}^*(\bar z)
\widehat{H}_{2,l}^{-1}Q_{2,l}(0), \quad j\geq1, \label{B12j}
\\
 P_{2,j}^*(\bar z) Q_{2,j}^{*^{-1}}(0)=& \sum_{l=0}^{j} P_{1,l}^*(\bar z)
 \widehat H_{1,l}^{-1}P_{1,l}(0),\quad j\geq0, \label{B21j}
\end{align}
and
\begin{equation}
 P_{1,j}^*(\bar z) P_{1,j}^{*^{-1}}(0)=I_q-z\sum_{l=0}^{j-1} P_{2,l}^*(\bar z)
 \widehat{H}_{2,l}^{-1}Q_{2,l}(0), \quad j\geq1.\label{B22j}
\end{equation}
\end{proposition}

\vskip1mm

The following result is fundamental for the developments in this paper.
 It provides a representation of an HTM polynomial \(\mathbf f_n\) in
terms of the matrix polynomials
\(P_{1,j}\), \(Q_{1,j}\), \(P_{2,j}\), and \(Q_{2,j}\).
This representation first appeared in \cite{abH}
(line~4 from the bottom of page~78 and line~8 from the bottom of
page~79).

\begin{proposition}\label{prop1.8}
For $n\geq1$, the matrix polynomial  $\f_n$ is an HTM polynomial if and only if $\f_n$ has the following
 representation:
 \begin{equation}
 \f_n(z)=\left\{
 \begin{array}{cl}
   (-1)^m(P_{1,m}^*(-\bar z^2)-z\, Q_{1,m}^*(-\bar z^2)), & n=2m, \\
  (-1)^m(Q_{2,m}^*(-\bar z^2)+z\,  P_{2,m}^*(-\bar z^2)), & n=2m+1,
 \end{array}
 \right. \label{hupfqp}
\end{equation}
where $P_{k,m}$ and  $Q_{k,m}$ for $k=1,2$ are polynomials defined in Definition \ref{def2p}.
The matrix polynomials
$P_{k,m}$ and  $Q_{k,m}$
are constructed from the Markov parameters associated with \(\mathbf f_n\).
\end{proposition}


\section{Relations between Markov parameters and coefficients}\label{sec3}
Let \(H_{1,j}\) and \(H_{2,j}\) be the block Hankel matrices defined
in \eqref{61}, and let \(Y_{r,j}\), \(r=1,2\), be defined by
\eqref{81}. Since the underlying moment sequence is Stieltjes
positive, these Hankel matrices are Hermitian positive definite.

The following lemma provides explicit formulas for the coefficients of
an HTM polynomial in terms of the matrices
\(\Sigma_{r,j}\), \(r=1,2\), which are constructed from the Markov
parameters. For \(r=1,2\), define
\begin{equation}\label{sigrr}
\Sigma_{r,j}:=
\begin{pmatrix}
-H_{r,j-1}^{-1}Y_{r,j}\\
I_q
\end{pmatrix},
\qquad j\geq1.
\end{equation}
These formulas are obtained by expanding the representation
\eqref{hupfqp} in powers of \(z\).

\begin{lemma}\label{pol-3}
Let \(A_j\) be the coefficients of the matrix polynomial
\(\mathbf f_n\) in \eqref{hup01}. Furthermore, let
\(T_j\), \(v_j\), \(u_{k,j}\), and \(\Sigma_{k,j}\), for \(k=1,2\),
be as in \eqref{56}, \eqref{60a}, \eqref{66}, and \eqref{sigrr},
respectively. 
Then the coefficients of \(\mathbf f_n\) admit the following explicit
representations.
If \(n=2m\), then
      \begin{align}
  A_{2j}=&(-1)^{m+j} v_m^* (T_m^{m-j})^*\Sigma_{1,m}, \quad j=0,1,\ldots, m-1,m, \label{A2m2mj0}\\
    A_{2j+1}=&(-1)^{m+j} u_{1,m}^* (T_m^{m-j})^*\Sigma_{1,m}, \quad j=0,1,\ldots, m-1. \label{A2m2mj10}
     \end{align}
     If \(n=2m+1\), then
 \begin{align}
  A_{2j}=&(-1)^{m+j} v_m^* (T_m^{m-j})^*\Sigma_{2,m}, \quad j=0,1,\ldots, m-1,m, \label{A2m2mj0a}\\
    A_{2j+1}=&(-1)^{m+j+1} u_{2,m}^*(T_m^{m-j})^*\Sigma_{2,m}, \quad j=0,1,\ldots, m-1,m. \label{A2m2mj10a}
     \end{align}
     In particular, the following equalities hold.  
     For \(n=2m\),
   \begin{align}
I_q&=   v_m^* (T_m^m)^*\Sigma_{1,m},\label{c01}\\
A_{1}&=-u_{1,m}^*(T_m^{m-1})^*\Sigma_{1,m}=s_0,\label{c02}\\
A_{2m-1}&=(-1)^m u_{1,m}^*\Sigma_{1,m}=(-1)^{m+1} Q_{1,m}^*(0),\label{c03}\\
A_{2m}&=(-1)^m v_m^*\Sigma_{1,m}=(-1)^m P_{1,m}^*(0).\label{c04}
   \end{align}
For \(n=2m+1\), the following equalities hold.
   \begin{align}
I_q&= v_m^* (T_m^m)^*\Sigma_{2,m},\label{c11}\\
A_{1}&=-u_{2,m}^*(T_m^m)^*\Sigma_{2,m}=s_0,\label{c12}\\
A_{2m}&= (-1)^{m} v_{m}^*\Sigma_{2,m}=(-1)^m P_{2,m}^*(0),\label{c13}\\
A_{2m+1}&= (-1)^{m+1} u_{2,m}^*\Sigma_{2,m}=(-1)^m Q_{2,m}^*(0).\label{c14}
   \end{align}
   \end{lemma}
\begin{proof}
From Proposition~\ref{prop1.8}, together with
(\ref{p1j})--(\ref{qj2}),
we obtain
 \begin{align}
 \f_{2m}(z)
 =&(-1)^m\left(v_m^*(T_m^{m})^*z^{2m}-u_{1,m}^*(T_m^{m-1})^*z^{2m-1}-v_m^*(T_m^{m-1})^*z^{2m-2}+\ldots \right.\nonumber\\
 & + (-1)^{m+1} u_{1,m}^*T_m^*z^3
 + (-1)^{m+1}z^2v_{m}^*T_m^*+(-1)^{m}z u_{1,m}^*\nonumber\\
 &\left.+(-1)^mv_m^*
 \right)\Sigma_{1,m}, \label{f2mAA}
 \\
\f_{2m+1}(z)=&
(-1)^m\left(v_m^*(T_m^m)^*z^{2m+1}-u_{2,m}^*(T_m^m)^*z^{2m}-v_m^*(T_m^{m-1})^*z^{2m-1}+\ldots
\right. \nonumber\\
& +(-1)^{m+1}  v_m^*T_m^*z^3  +(-1)^m u_{2,m}^* T_m^* z^2+(-1)^m v_m^* z\nonumber\\
  &\left.+ (-1)^{m+1} u_{2,m}^*
 \right)\Sigma_{2,m}. \label{f2mp1AA}
 \end{align}
 Comparing the coefficients of the same powers of \(z\)  in the above
representations and in (\ref{hup01}) yields
(\ref{A2m2mj0})--(\ref{A2m2mj10a}).

Identities (\ref{c01}) and (\ref{c11}) follow immediately from
(\ref{56}), (\ref{60a}), and the definition of \(\Sigma_{r,m}\).
The first equalities in (\ref{c02})--(\ref{c04}) and (\ref{c12})--(\ref{c14})
follow by comparing the coefficients of the corresponding powers of
\(z\). The second equalities in (\ref{c03})--(\ref{c04}) and (\ref{c13})--(\ref{c14})
follow immediately from the representation (\ref{hupfqp}) in Proposition~\ref{prop1.8}.
\end{proof}
 The explicit representations obtained in Lemma~\ref{pol-3} can be
worked out completely for low-degree HTM polynomials. The resulting
formulas illustrate the dependence of the polynomial coefficients on
the associated Markov parameters and will be used later in several
examples.
\begin{remark}\label{rem4.1}
Let \((s_j)_{j=0}^{2\kappa}\) (resp. \((s_j)_{j=0}^{2\kappa+1}\)) be Stieltjes positive sequences for
\(\kappa=0,1,2\).
Applying Lemma~\ref{pol-3} yields the following explicit expressions for the HTM polynomials \(\mathbf f_j\),
\(j=1,\ldots,5\),
in terms of the corresponding Markov parameters.
\begin{align}
\f_1(z)=&I_q z + s_0,\nonumber\\
\f_2(z)=&I_q z^2 + s_0z+ s_1s_0^{-1}, \nonumber\\
\f_3(z)=&I_q z^3  + s_0 z^2 +s_1^{-1}\,s_2 z  + s_0s_1^{-1}s_2  -s_1,\nonumber\\
\f_4(z)=&I_q z^4  + s_0 z^3 \nonumber\\
&+(s_2-s_1s_0^{-1}s_1)^{-1}(s_3-s_1s_0^{-1}s_2)z^2\label{coefAA2}
\\
&
+(s_0
(s_2-s_1s_0^{-1}s_1)^{-1}(s_3-s_1s_0^{-1}s_2)-s_1
)z,\nonumber\\
&+s_0^{-1}s_1(s_2-s_1s_0^{-1}s_1)^{-1}(s_3-s_1s_0^{-1}s_2)-s_0^{-1}s_2\nonumber
\\
\f_5(z)=&I_q z^5  + s_0 z^4 
+\left(s_3 - s_2 s_1^{-1} s_2\right)^{-1}\left(s_4 - s_2 s_1^{-1} s_3\right)z^3\nonumber\\
&+
\left(s_0 \left(s_3 - s_2 s_1^{-1} s_2\right)^{-1}\left(s_4 - s_2 s_1^{-1} s_3\right) - s_1\right)z^2\nonumber\\
&+s_1^{-1}\left(s_2 \left(s_3 - s_2 s_1^{-1} s_2\right)^{-1}\left(s_4 - s_2 s_1^{-1} s_3\right) - s_3\right)z\nonumber\\
&+s_0 s_1^{-1}(s_2 \left(s_3 - s_2 s_1^{-1} s_2\right)^{-1}\left(s_4 - s_2 s_1^{-1} s_3\right) - s_3)  + s_2\nonumber\\
&- s_1 \left(s_3 - s_2 s_1^{-1} s_2\right)^{-1}\left(s_4 - s_2 s_1^{-1} s_3\right).\nonumber
\end{align}
\end{remark}
The previous remark provides explicit expressions for HTM polynomials
of degrees \(1,\ldots,5\) in terms of their Markov parameters.
Analogous formulas in the scalar case were obtained in
\cite[Remark~8]{ab2020A}.

 The following remark gives
the inverse correspondence by expressing the Markov parameters in
terms of the polynomial coefficients.
\begin{remark}\label{rem4.101}
For degrees \(n=1,\ldots,5\), the Markov parameters are given by
\begin{align*}
&\text{for } n=1,\quad s_0=A_1,\\
&\text{for } n=2,\quad s_0=A_1,\quad s_1=A_2 A_1,\\
&\text{for } n=3,\quad s_0=A_1,\quad s_1=A_1A_2-A_3,\quad
s_2=(A_1A_2-A_3)A_2,\\
&\text{for } n=4,\quad s_0=A_1,\quad s_1=A_1A_2-A_3,\quad
s_2=(A_1A_2-A_3)A_2-A_1A_4,\\
&\qquad
 s_3=(A_3-A_1A_2)A_4+((A_1A_2-A_3)A_2-A_1A_4)A_2,\\
&\text{for } n=5,\quad s_0=A_1,\quad s_1=A_1A_2-A_3,\quad
s_2=A_5-A_1A_4+(A_1A_2-A_3)A_2,\\
&\qquad s_3=(A_5-A_1A_4+(A_1A_2-A_3)A_2)A_2
+(A_3-A_1A_2)A_4,\\
&\qquad s_4=((A_5-A_1A_4+(A_1A_2-A_3)A_2)A_2
+(A_3-A_1A_2)A_4)A_2\\
&\qquad\qquad
-(A_5-A_1A_4+(A_1A_2-A_3)A_2)A_4.
\end{align*}
\end{remark}


\section{Representation of the coefficients \(A_j\) in terms of orthogonal matrix polynomials}\label{sec4}
In this section, we derive explicit representations for the coefficients
\(A_j\) of an HTM polynomial in terms of the orthogonal matrix
polynomials \(P_{k,m}\) and \(Q_{k,m}\), their derivatives at the
origin, and the associated Schur complements.

The result follows directly from the representation
\eqref{hupfqp} by expanding the involved matrix polynomials into
Taylor series at the origin.
\begin{lemma}\label{lem3.14}
Let \(A_0,\ldots,A_n\) be the coefficients of the HTM polynomial
\(\mathbf f_n\) defined by \eqref{hup01}, and suppose that
\(\mathbf f_n\) admits the representation \eqref{hupfqp}.
Furthermore, let \(P_{k,m}\) and \(Q_{k,m}\), \(k=1,2\), be as in
Definition~\ref{def2p}. Then the following identities hold.

If \(n=2m\), \(m\geq1\), then
\begin{equation}\label{coef1}
A_{n-k}
=
\begin{cases}
\displaystyle
\frac{(-1)^{m+j}}{j!}
\left(P_{1,m}^{*}\right)^{(j)}(0),
&
k=2j,\quad 0\leq j\leq m,
\\[2ex]
\displaystyle
-\frac{(-1)^{m+j}}{j!}
\left(Q_{1,m}^{*}\right)^{(j)}(0),
&
k=2j+1,\quad 0\leq j\leq m-1.
\end{cases}
\end{equation}

If \(n=2m+1\), \(m\geq0\), then
\begin{equation}\label{coef2}
A_{n-k}
=
\begin{cases}
\displaystyle
\frac{(-1)^{m+j}}{j!}
\left(Q_{2,m}^{*}\right)^{(j)}(0),
&
k=2j,\quad 0\leq j\leq m,
\\[2ex]
\displaystyle
\frac{(-1)^{m+j}}{j!}
\left(P_{2,m}^{*}\right)^{(j)}(0),
&
k=2j+1,\quad 0\leq j\leq m.
\end{cases}
\end{equation}
\end{lemma}

\begin{proof}
From \eqref{hup01} and Taylor's formula,
\begin{equation}\label{eq:coeff-derivative}
A_{n-k}
=
\frac{1}{k!}\mathbf f_n^{(k)}(0),
\qquad
k=0,1,\ldots,n.
\end{equation}

Suppose first that \(n=2m\). Expanding the representation
\eqref{hupfqp} at the origin gives
\[
\mathbf f_{2m}(z)
=
\sum_{j=0}^{m}
\frac{(-1)^{m+j}}{j!}
\left(P_{1,m}^{*}\right)^{(j)}(0)z^{2j}
-
\sum_{j=0}^{m-1}
\frac{(-1)^{m+j}}{j!}
\left(Q_{1,m}^{*}\right)^{(j)}(0)z^{2j+1}.
\]
Comparing the coefficients of \(z^{2j}\) and \(z^{2j+1}\) with
\eqref{hup01}, or equivalently using
\eqref{eq:coeff-derivative}, yields \eqref{coef1}.

Suppose next that \(n=2m+1\). Then
\[
\mathbf f_{2m+1}(z)
=
\sum_{j=0}^{m}
\frac{(-1)^{m+j}}{j!}
\left(Q_{2,m}^{*}\right)^{(j)}(0)z^{2j}
+
\sum_{j=0}^{m}
\frac{(-1)^{m+j}}{j!}
\left(P_{2,m}^{*}\right)^{(j)}(0)z^{2j+1}.
\]
Comparing coefficients with \eqref{hup01} gives \eqref{coef2}.
\end{proof}

\subsection{Nested sums and auxiliary matrices}
To simplify the notation, define
\begin{equation}\label{PPQQ0}
{\mathcal P}_{1,r}^{(0)}
=
P_{1,r}^{*}(0)\widehat H_{1,r}^{-1}P_{1,r}(0),
\qquad
{\mathcal Q}_{2,r}^{(0)}
=
Q_{2,r}^{*}(0)\widehat H_{2,r}^{-1}Q_{2,r}(0).
\end{equation}
Furthermore, let
\begin{align}
\Xi_m^{(0)}:=&I_q,\qquad
\Phi_m^{(0)}:=I_q,\qquad
\Theta_m^{(0)}:=\sum_{l=0}^{m}{\mathcal P}_{1,l}^{(0)},\qquad m\ge0,\label{01-eq}
\\
\Psi_0^{(0)}:=&0_q,
\qquad
\Psi_m^{(0)}
:=
\sum_{l=0}^{m-1}{\mathcal Q}_{2,l}^{(0)},
\quad m\ge1,
\label{11-eq}
\end{align}
where an empty sum is understood as the zero matrix \(0_{q\times q}\).

\begin{remark}\label{rem4.1A}
Let \(P_{k,j}\) and \(Q_{k,j}\) be as in
Definition~\ref{def2p}. 
Setting \(z=0\) in
\eqref{B12j} and \eqref{B21j}, and using the definitions
\eqref{PPQQ0}, \eqref{01-eq}, and \eqref{11-eq}, we obtain
\begin{align}
Q_{1,m}^*(0)
&=
-\Psi_m^{(0)}P_{1,m}^*(0),
\qquad m\geq1,
\label{QQ1J}
\\
P_{2,m}^*(0)
&=
\Theta_m^{(0)}Q_{2,m}^*(0),
\qquad m\geq0.
\label{PP2J}
\end{align}
\end{remark}

For \(r,j\in\mathbb N_0\), define
\begin{equation}\label{QQeq2}
\begin{aligned}
{\mathcal P}_{1,r}^{(j)}
&:=
\left(P_{1,r}^{(j)}(0)\right)^*
\widehat H_{1,r}^{-1}
P_{1,r}(0),
\\
{\mathcal Q}_{1,r}^{(j)}
&:=
\left(Q_{1,r}^{(j)}(0)\right)^*
\widehat H_{1,r}^{-1}
P_{1,r}(0),
\\
{\mathcal P}_{2,r}^{(j)}
&:=
\left(P_{2,r}^{(j)}(0)\right)^*
\widehat H_{2,r}^{-1}
Q_{2,r}(0),
\\
{\mathcal Q}_{2,r}^{(j)}
&:=
\left(Q_{2,r}^{(j)}(0)\right)^*
\widehat H_{2,r}^{-1}
Q_{2,r}(0),
\end{aligned}
\qquad r,j\in\mathbb N_0.
\end{equation}

\begin{lemma}\label{lem4.30m}
Let \(P_{k,j}\) and \(Q_{k,j}\) be as in Definition~\ref{def2p}.
Let \(\widehat H_{k,j}\), \(k=1,2\), be as in
\eqref{hh1j}--\eqref{hh2j}, 
and let
\({\mathcal P}_{1,r}^{(j)}\),
\({\mathcal Q}_{1,r}^{(j)}\),
\({\mathcal P}_{2,r}^{(j)}\), and
\({\mathcal Q}_{2,r}^{(j)}\)
be as in \eqref{QQeq2}.
Then, for every \(j\ge1\),
\begin{align}
\left(Q_{2,m}^{(j)}(0)\right)^*=&-j\left(
\sum_{l=0}^{m}
\left(
\sum_{r=0}^{l-1}
{\mathcal Q}_{2,r}^{(j-1)}
\right)
{\mathcal P}_{1,l}^{(0)}
\right)
Q_{2,m}^*(0), \qquad m\geq0
\label{mB11j0}
\\
\left(P_{1,m}^{(j)}(0)\right)^*=&-j\left(\sum_{l=0}^{m-1}\left(\sum_{r=0}^{l}
{\mathcal P}_{1,r}^{(j-1)}\right){\mathcal Q}_{2,l}^{(0)}\right)P_{1,m}^*(0), \qquad m\geq1,
\label{mB22j0}
\\
\left(Q_{1,m}^{(j)}(0)\right)^*=&-j\left(\sum_{l=0}^{m-1}\left(\sum_{r=0}^{l}{\mathcal Q}_{1,r}^{(j-1)}\right)
{\mathcal Q}_{2,l}^{(0)}\right)P_{1,m}^*(0),
\qquad m\geq1,
\label{mB12j0}
\end{align}
and
\begin{equation}
\left(P_{2,m}^{(j)}(0)\right)^*=-j
\left(
\sum_{l=0}^{m}
\left(
\sum_{r=0}^{l-1}
{\mathcal P}_{2,r}^{(j-1)}
\right)
{\mathcal P}_{1,l}^{(0)}
\right)
Q_{2,m}^*(0), \qquad m\geq0.
\label{mB21j0}
\end{equation}
\end{lemma}
 \begin{proof}
 The identities follow by differentiating
\eqref{B11j}--\eqref{B22j}
\(j\) times with respect to \(z\), evaluating at
\(z=0\), using
\[
\left.
\frac{d^j}{dz^j}\bigl(zF(z)\bigr)
\right|_{z=0}
=
jF^{(j-1)}(0),
\]
and finally applying the definitions in
\eqref{QQeq2}.
\end{proof}
The following auxiliary matrices will be used to express the iterated
identities obtained from Lemma~\ref{lem4.30m}. For \(j\ge1\), define
\begin{align}
\Xi_m^{(j)}&:=\sum_{l=0}^{m}\Psi_l^{(j-1)}{\mathcal P}_{1,l}^{(0)},
\qquad
\Phi_m^{(j)}:=\sum_{l=0}^{m-1}\Theta_l^{(j-1)}{\mathcal Q}_{2,l}^{(0)},
\label{04-eq}
\\
\Theta_m^{(j)}&:=\sum_{l=0}^{m}\Phi_l^{(j)}{\mathcal P}_{1,l}^{(0)}, \qquad
\Psi_m^{(j)}:=\sum_{l=0}^{m-1}\Xi_l^{(j)}{\mathcal Q}_{2,l}^{(0)}.\label{06-eq}
\end{align}

The next lemma follows by iteratively applying \eqref{mB11j0}--\eqref{mB21j0}
together with the recursive definitions \eqref{04-eq}--\eqref{06-eq}.
\begin{lemma}\label{lem3.3}
Let \(P_{k,m}\) and \(Q_{k,m}\) be as in
Definition~\ref{def2p}, 
and let $\Xi_m^{(j)}$, $\Phi_m^{(j)}$, $\Theta_m^{(j)}$, and $\Psi_m^{(j)}$ be as in
\eqref{04-eq} and \eqref{06-eq}.
Then, for every \(j\ge1\), the following identities hold.
\begin{align}
\left(Q_{2,m}^{(j)}(0)\right)^*&=(-1)^j j!\Xi_m^{(j)}Q_{2,m}^*(0),\qquad m\geq0,
\label{nn11j0}
\\
\left(P_{1,m}^{(j)}(0)\right)^*&=(-1)^j j! \Phi_m^{(j)}P_{1,m}^{*}(0),
\qquad m\geq1,\label{nn22j0}
\\
\left(Q_{1,m}^{(j)}(0)\right)^*&=(-1)^{j+1}j!\Psi_m^{(j)}P_{1,m}^{*}(0),
\qquad m\geq1,
\label{nn12j0}
\end{align}
and
\begin{equation}
\left(P_{2,m}^{(j)}(0)\right)^*=(-1)^j j!\Theta_m^{(j)}Q_{2,m}^{*}(0),\qquad m\geq0.
\label{nn21j0}
\end{equation}
\end{lemma}
\begin{proof}
The proof proceeds by induction on \(j\).

For \(j=1\), the identities follow immediately from
\eqref{mB11j0}--\eqref{mB21j0} together with the definitions
\eqref{04-eq}--\eqref{06-eq}.

Suppose that the assertions hold for \(j-1\).
Then
\[
{\mathcal Q}_{2,r}^{(j-1)}=(-1)^{j-1}(j-1)! \Xi_r^{(j-1)}{\mathcal Q}_{2,r}^{(0)}.\]
Using the induction hypothesis and \eqref{mB11j0}, we obtain
\begin{align*}
\left(Q_{2,m}^{(j)}(0)\right)^*&=-j(-1)^{j-1}(j-1)!\left[\sum_{l=0}^{m}\left(\sum_{r=0}^{l-1}
\Xi_r^{(j-1)}{\mathcal Q}_{2,r}^{(0)}
\right)
{\mathcal P}_{1,l}^{(0)}
\right]
Q_{2,m}^*(0)
\\
&=(-1)^j j!\left[\sum_{l=0}^{m}\Psi_l^{(j-1)}{\mathcal P}_{1,l}^{(0)}\right]
Q_{2,m}^*(0)
\\
&=(-1)^j j! \Xi_m^{(j)}Q_{2,m}^*(0).
\end{align*}
The remaining identities are obtained in the same way.
Indeed,
${\mathcal P}_{1,r}^{(j-1)}=(-1)^{j-1}(j-1)!\Phi_r^{(j-1)}{\mathcal P}_{1,r}^{(0)}$,
and \eqref{mB22j0} gives \eqref{nn22j0}.
Furthermore,
${\mathcal Q}_{1,r}^{(j-1)}=(-1)^j(j-1)!\Psi_r^{(j-1)}{\mathcal P}_{1,r}^{(0)}$, 
so that \eqref{mB12j0} yields \eqref{nn12j0}.
Finally,
${\mathcal P}_{2,r}^{(j-1)}=(-1)^{j-1}(j-1)!\Theta_r^{(j-1)}{\mathcal Q}_{2,r}^{(0)}$,
and \eqref{mB21j0} implies \eqref{nn21j0}.
This completes the proof.
\end{proof}

The next remark shows that the matrices \({\mathcal P}_{1,j}^{(0)}\) and
\({\mathcal Q}_{2,j}^{(0)}\) can be written explicitly in terms of the Schur complements.
\begin{remark}\label{rem3.4}
Let \({\mathcal P}_{1,j}^{(0)}\), \({\mathcal Q}_{2,j}^{(0)}\),
and \(\widehat H_{k,j}\) be as above.

By Proposition~4.9(e) of \cite{ablaa},
\begin{align}
P_{1,j}^*(0)&=(-1)^j\widehat H_{1,0}^{-1}\widehat H_{2,0}\cdots \widehat H_{1,j-1}^{-1}\widehat H_{2,j-1},
\label{eq:P1star0}
\\
Q_{2,j}^*(0)
&=
(-1)^j
\widehat H_{1,0}\widehat H_{2,0}^{-1}
\widehat H_{1,1}
\cdots
\widehat H_{2,j-1}^{-1}\widehat H_{1,j}.
\label{eq:Q2star0}
\end{align}
Consequently,
\begin{align}
{\mathcal Q}_{2,j}^{(0)}
&=
\widehat H_{1,0}\widehat H_{2,0}^{-1}
\widehat H_{1,1}
\cdots
\widehat H_{2,j-1}^{-1}\widehat H_{1,j}
\widehat H_{2,j}^{-1}
\widehat H_{1,j}\widehat H_{2,j-1}^{-1}
\cdots
\widehat H_{1,1}\widehat H_{2,0}^{-1}
\widehat H_{1,0},
\label{eq:Q20-explicit}
\\
{\mathcal P}_{1,j}^{(0)}
&=
\widehat H_{1,0}^{-1}\widehat H_{2,0}
\cdots
\widehat H_{1,j-1}^{-1}\widehat H_{2,j-1}
\widehat H_{1,j}^{-1}
\widehat H_{2,j-1}\widehat H_{1,j-1}^{-1}
\cdots
\widehat H_{2,0}\widehat H_{1,0}^{-1},
\label{eq:P10-explicit}
\end{align}
where the products are interpreted as empty products when \(j=0\).
\end{remark}

The following theorem is the main result of this section.
It shows that every coefficient of an HTM polynomial admits an explicit
representation solely in terms of the associated Schur complements.
\begin{theorem}\label{teo3.4}
Let \(\mathbf f_n\) be the matrix polynomial \eqref{hup01}, and let
\(\Xi_m^{(j)}\), \(\Phi_m^{(j)}\),
\(\Psi_m^{(j)}\), and \(\Theta_m^{(j)}\)
be as in \eqref{01-eq}, \eqref{11-eq},
\eqref{04-eq}, and \eqref{06-eq}.
Then the coefficients of \(\mathbf f_n\) admit the following
representations.
For \(n=2m\), \(m\geq1\),
\begin{align}\label{coef1CC}
A_{n-k}
=
\begin{cases}
\displaystyle
\Phi_m^{(j)}
\widehat H_{1,0}^{-1}\widehat H_{2,0}
\cdots
\widehat H_{1,m-1}^{-1}\widehat H_{2,m-1},
&
\text{if } k=2j,\quad 0\leq j\leq m,
\\[2ex]
\displaystyle
\Psi_m^{(j)}
\widehat H_{1,0}^{-1}\widehat H_{2,0}
\cdots
\widehat H_{1,m-1}^{-1}\widehat H_{2,m-1},
&
\text{if } k=2j+1,\quad 0\leq j\leq m-1.
\end{cases}
\end{align}

For \(n=2m+1\), \(m\geq0\),
\begin{align}\label{coef2CC}
A_{n-k}
=
\begin{cases}
\displaystyle
\Xi_m^{(j)}
\widehat H_{1,0}\widehat H_{2,0}^{-1}
\widehat H_{1,1}
\cdots
\widehat H_{2,m-1}^{-1}\widehat H_{1,m},
&
\text{if } k=2j,\quad 0\leq j\leq m,
\\[2ex]
\displaystyle
\Theta_m^{(j)}
\widehat H_{1,0}\widehat H_{2,0}^{-1}
\widehat H_{1,1}
\cdots
\widehat H_{2,m-1}^{-1}\widehat H_{1,m},
&
\text{if } k=2j+1,\quad 0\leq j\leq m.
\end{cases}
\end{align}
For \(m=0\), there are no factors following
\(\widehat H_{1,0}\); hence the product appearing in
\eqref{coef2CC} is understood to be simply
\(\widehat H_{1,0}\).
\end{theorem}

\begin{proof}
Combining \eqref{coef1}, \eqref{coef2}, and \eqref{nn11j0}--\eqref{nn21j0},
 we obtain, for \(n=2m\),
\begin{align}\label{coef1CC00}
A_{n-k}
=
\begin{cases}
\displaystyle
(-1)^m\Phi_m^{(j)}P_{1,m}^*(0),
& \text{if } k=2j,
\\[2ex]
\displaystyle
(-1)^m\Psi_m^{(j)}P_{1,m}^*(0),
& \text{if } k=2j+1,
\end{cases}
\end{align}
and, for \(n=2m+1\),
\begin{align}\label{coef2CC00}
A_{n-k}
=
\begin{cases}
\displaystyle
(-1)^m\Xi_m^{(j)}Q_{2,m}^*(0),
& \text{if } k=2j,
\\[2ex]
\displaystyle
(-1)^m\Theta_m^{(j)}Q_{2,m}^*(0),
& \text{if } k=2j+1.
\end{cases}
\end{align}
Finally, substituting \eqref{eq:P1star0} and
\eqref{eq:Q2star0} into \eqref{coef1CC00} and
\eqref{coef2CC00}, respectively, produces
\eqref{coef1CC} and \eqref{coef2CC}.
\end{proof}

\begin{remark}\label{rem4.03AAA}
Theorem~\ref{teo3.4} eliminates the orthogonal matrix polynomials
\(P_{k,m}\) and \(Q_{k,m}\) from the representation of the
coefficients.
Consequently, every coefficient of an HTM polynomial can be computed
directly from the Schur complements
\(\widehat H_{1,j}\) and
\(\widehat H_{2,j}\).
\end{remark}

\subsection{Low-degree examples}

In this subsection we present explicit expressions for the HTM
polynomials of degrees \(1,\ldots,5\). These examples illustrate the
general formulas obtained in the previous subsection and admit several
equivalent representations. We first express the polynomials in terms
of the orthogonal matrix polynomials
\((P_{k,j},Q_{k,j})\), then in terms of the auxiliary matrices
\({\mathcal P}_{1,j}^{(0)}\) and
\({\mathcal Q}_{2,j}^{(0)}\), next in terms of the Schur complements
\(\widehat H_{1,j}\) and \(\widehat H_{2,j}\), and finally in terms of
the Hurwitz parameters \((\cc_j,\dd_j)\).

\begin{remark}\label{rem4.3}
Let $P_{k,j}$ and $Q_{k,j}$ be the matrix polynomials defined in Definition~\ref{def2p}.
Then the HTM polynomials \(f_j\), \(j=1,\ldots,5\), are given
in terms of \(P_{k,j}\), \(Q_{k,j}\), and their derivatives by
the following equalities.
\begin{align*}
\f_1(z)=&
P_{2,0}^*(0)z+Q_{2,0}^*(0)\\
\f_2(z)=&-P_{1,1}^*(0)+zQ_{1,1}^*(0)+z^2\left(P_{1,1}^\prime(0)\right)^*\\
\f_3(z)=&-Q_{2,1}^*(0)-P_{2,1}^*(0)z+ (Q_{2,1}^\prime(0))^* z^2+(P_{2,1}^\prime(0))^* z^3\\
\f_4(z)=&P_{1,2}^*(0)-(Q_{1,2}^*(0))z- (P_{1,2}^\prime(0))^* z^2+(Q_{1,2}^\prime(0))^* z^3+\frac{1}{2}(P_{1,2}^{\prime\prime}(0))^* z^4\\
\f_5(z)=&Q_{2,2}^*(0)+P_{2,2}^*(0)z-(Q_{2,2}^\prime(0))^* z^2-(P_{2,2}^\prime(0))^* z^3+\frac{1}{2}(Q_{2,2}^{\prime\prime}(0))^* z^4\\
&+
\frac{1}{2}(P_{2,2}^{\prime\prime}(0))^* z^5
\end{align*}
\end{remark}

The representations in Remark~\ref{rem4.3} still involve the
orthogonal matrix polynomials and their derivatives. Using the
identities established in the previous subsection, these derivatives
can be eliminated in favor of the auxiliary matrices
\({\mathcal P}_{1,j}^{(0)}\) and
\({\mathcal Q}_{2,j}^{(0)}\), leading to the following formulas.

\begin{remark}\label{rem4.3A}
Let \({\mathcal Q}_{2,l}^{(0)}\) and
\({\mathcal P}_{1,l}^{(0)}\) be as in \eqref{PPQQ0}.
Furthermore, let \(P_{k,j}\) and \(Q_{k,j}\) be as in
Definition~\ref{def2p}.
Using \eqref{QQeq2}, \eqref{QQ1J}, \eqref{PP2J},
and Lemma~\ref{lem4.30m},
the HTM polynomials \(\mathbf f_j\), \(j=1,\ldots,5\),
admit the following representations in terms of
\({\mathcal Q}_{2,l}^{(0)}\) and
\({\mathcal P}_{1,l}^{(0)}\).
\begin{align}
\f_1(z)=&\left({\mathcal P}_{1,0}^{(0)}z+I\right)Q_{2,0}^*(0)\nonumber
\\
\f_2(z)=&-\left({\mathcal P}_{1,0}^{(0)}{\mathcal Q}_{2,0}^{(0)}z^2+{\mathcal Q}_{2,0}^{(0)}z+I\right)P_{1,1}^*(0)
\nonumber
\\
\f_3(z)=&-\left({\mathcal P}_{1,0}^{(0)} {\mathcal Q}_{2,0}^{(0)}{\mathcal P}_{1,1}^{(0)} z^3+{\mathcal Q}_{2,0}^{(0)}{\mathcal P}_{1,1}^{(0)}z^2+
({\mathcal P}_{1,0}^{(0)}+{\mathcal P}_{1,1}^{(0)})z+I\right)Q_{2,1}^*(0)
\nonumber
\\
\f_4(z)=&
 \left({\mathcal P}_{1,0}^{(0)} {\mathcal Q}_{2,0}^{(0)}{\mathcal P}_{1,1}^{(0)}{\mathcal Q}_{2,1}^{(0)} z^4
 +{\mathcal Q}_{2,0}^{(0)}{\mathcal P}_{1,1}^{(0)}{\mathcal Q}_{2,1}^{(0)}z^3\right.\nonumber\\
 &+\left.
({\mathcal P}_{1,0}^{(0)}{\mathcal Q}_{2,0}^{(0)}+\left({\mathcal P}_{1,0}^{(0)}+{\mathcal P}_{1,1}^{(0)}\right){\mathcal Q}_{2,1}^{(0)})z^2
+\left({\mathcal Q}_{2,0}^{(0)}+{\mathcal Q}_{2,1}^{(0)}\right)z+I\right)P_{1,2}^*(0)
\nonumber
\\
\f_5(z)=
&\left(
             {\mathcal P}_{1,0}^{(0)}{\mathcal Q}_{2,0}^{(0)}
             {\mathcal P}_{1,1}^{(0)}{\mathcal Q}_{2,1}^{(0)}
              {\mathcal P}_{1,2}^{(0)}z^5
+{\mathcal Q}_{2,0}^{(0)}{\mathcal P}_{1,1}^{(0)}{\mathcal Q}_{2,1}^{(0)}{\mathcal P}_{1,2}^{(0)}z^4\right.
\nonumber
\\
&
+
\left({\mathcal P}_{1,0}^{(0)}{\mathcal Q}_{2,0}^{(0)}{\mathcal P}_{1,1}^{(0)}
+
\left({\mathcal P}_{1,0}^{(0)}{\mathcal Q}_{2,0}^{(0)}+\left({\mathcal P}_{1,0}^{(0)}+{\mathcal P}_{1,1}^{(0)}\right){\mathcal Q}_{2,1}^{(0)}
\right){\mathcal P}_{1,2}^{(0)}\right)z^3
\nonumber
\\
&
\left.
+\left({\mathcal Q}_{2,0}^{(0)}{\mathcal P}_{1,1}^{(0)}+\left({\mathcal Q}_{2,0}^{(0)}+{\mathcal Q}_{2,1}^{(0)}\right){\mathcal P}_{1,2}^{(0)}\right)z^2+
\left({\mathcal P}_{0,1}^{(0)}+{\mathcal P}_{1,1}^{(0)}+{\mathcal P}_{1,2}^{(0)}\right)z+I\right)\nonumber
\\
&\cdot Q_{2,2}^*(0).\label{f5PP}
\end{align}
\end{remark}
Since the matrices
\({\mathcal P}_{1,j}^{(0)}\) and
\({\mathcal Q}_{2,j}^{(0)}\) admit explicit representations in terms
of the Schur complements, the preceding identities can be rewritten
entirely in terms of the Schur complements
\(\widehat H_{1,j}\) and \(\widehat H_{2,j}\).

The corresponding formulas in the scalar case for polynomials of
degrees \(1,\ldots,5\) were established in
\cite[Remark~10]{ab2020A}. The following remark extends these
representations to the matrix setting.
\begin{remark}\label{rem4.2}
Let the Schur complements \(\widehat H_{1,j}\) and
\(\widehat H_{2,j}\) be defined as in \eqref{hh1j} and
\eqref{hh2j}, respectively. Then, by Remark~\ref{rem4.3}
and Lemma~\ref{lem3.3}, the HTM polynomials
\(\mathbf f_n\), \(n=1,\ldots,5\), admit the following
representations:
\begin{align}
\mathbf f_1(z)
&=
I_qz+\widehat H_{1,0},\nonumber
\\
\mathbf f_2(z)
&=
I_qz^2+\widehat H_{1,0}z
+\widehat H_{1,0}^{-1}\widehat H_{2,0},\nonumber
\\
\mathbf f_3(z)
&=
I_qz^3+\widehat H_{1,0}z^2
+\left(
\widehat H_{2,0}^{-1}\widehat H_{1,1}
+\widehat H_{1,0}^{-1}\widehat H_{2,0}
\right)z\nonumber
\\
&\quad
+\widehat H_{1,0}\widehat H_{2,0}^{-1}
\widehat H_{1,1},\nonumber
\\
\mathbf f_4(z)
&=
I_qz^4+\widehat H_{1,0}z^3
+\left(
\widehat H_{1,1}^{-1}\widehat H_{2,1}
+\widehat H_{2,0}^{-1}\widehat H_{1,1}
+\widehat H_{1,0}^{-1}\widehat H_{2,0}
\right)z^2\nonumber
\\
&\quad
+\widehat H_{1,0}
\left(
\widehat H_{1,1}^{-1}\widehat H_{2,1}
+\widehat H_{2,0}^{-1}\widehat H_{1,1}
\right)z\nonumber
\\
&\quad
+\widehat H_{1,0}^{-1}\widehat H_{2,0}
\widehat H_{1,1}^{-1}\widehat H_{2,1},\label{HH-5}
\\
\mathbf f_5(z)
&=
I_qz^5+\widehat H_{1,0}z^4\nonumber
\\
&\quad
+\left(
\widehat H_{2,1}^{-1}\widehat H_{1,2}
+\widehat H_{1,1}^{-1}\widehat H_{2,1}
+\widehat H_{2,0}^{-1}\widehat H_{1,1}
+\widehat H_{1,0}^{-1}\widehat H_{2,0}
\right)z^3\nonumber
\\
&\quad
+\widehat H_{1,0}
\left(
\widehat H_{2,1}^{-1}\widehat H_{1,2}
+\widehat H_{1,1}^{-1}\widehat H_{2,1}
+\widehat H_{2,0}^{-1}\widehat H_{1,1}
\right)z^2\nonumber
\\
&\quad
+\left(
\widehat H_{2,0}^{-1}\widehat H_{1,1}
\widehat H_{2,1}^{-1}\widehat H_{1,2}
+\widehat H_{1,0}^{-1}\widehat H_{2,0}
\widehat H_{2,1}^{-1}\widehat H_{1,2}
\right.\nonumber
\\
&\hspace{4cm}\left.
+\widehat H_{1,0}^{-1}\widehat H_{2,0}
\widehat H_{1,1}^{-1}\widehat H_{2,1}
\right)
z\nonumber
\\
&\quad
+\widehat H_{1,0}\widehat H_{2,0}^{-1}
\widehat H_{1,1}\widehat H_{2,1}^{-1}
\widehat H_{1,2}.\nonumber
\end{align}
\end{remark}
Finally, expressing the Schur complements through the Hurwitz
parameters yields the classical Hurwitz parametrization of the HTM
polynomials.

The formulas obtained above will now be illustrated by two explicit
examples.
\section{Examples}\label{sec5}
In this section, we present two examples of HTM polynomials of
degrees \(4\) and \(5\). The first illustrates the representation of
the coefficients in terms of the Schur complements, whereas the second
illustrates their representation in terms of the auxiliary matrices
\({\mathcal P}_{1,j}^{(0)}\) and
\({\mathcal Q}_{2,j}^{(0)}\).
Since both polynomials are HTM polynomials, they are Hurwitz matrix
polynomials.

\begin{example}\label{exa1}
Consider the moment sequence
\[s_0=\begin{pmatrix}2&1\\1&1\end{pmatrix},
\quad
s_1=\begin{pmatrix}13&5\\5&2\end{pmatrix},
\quad
s_2=\begin{pmatrix}274&102\\102&38\end{pmatrix},
\quad
s_3=\begin{pmatrix}7014&2607\\2607&969\end{pmatrix}.
\]
A direct computation shows that the block Hankel matrices
\(H_{1,1}\) and \(H_{2,1}\) are Hermitian positive definite.
Hence, \((s_j)_{j=0}^{3}\) is a Stieltjes positive sequence.
The corresponding Schur complements are
\[
\widehat H_{1,0}=
\begin{pmatrix}
2&1\\
1&1
\end{pmatrix},
\quad
\widehat H_{2,0}=
\begin{pmatrix}
13&5\\
5&2
\end{pmatrix},
\quad
\widehat H_{1,1}
=
\begin{pmatrix}
185&68\\
68&25
\end{pmatrix},
\quad
\widehat H_{2,1}
=
\begin{pmatrix}
1090&403\\
403&149
\end{pmatrix}.
\]
Applying \eqref{HH-5}, the coefficients \(A_j\) of the HTM polynomial
\begin{equation}\label{fn4AA}
\mathbf f_4(z)
=
I_2z^4+A_1z^3+A_2z^2+A_3z+A_4
\end{equation}
are given by
\[
A_1=
\begin{pmatrix}
2&1\\
1&1
\end{pmatrix},
\qquad
A_2=
\begin{pmatrix}
-116&-43\\
391&145
\end{pmatrix},
\]
\[
A_3=
\begin{pmatrix}
146&54\\
270&100
\end{pmatrix},
\qquad
A_4=
\begin{pmatrix}
73&27\\
27&10
\end{pmatrix}.
\]
\end{example}

\begin{example}\label{exa2}
Let
$
w(x)=
\begin{pmatrix}
e^{-x}&-ie^{-x}&e^{-x}\\
ie^{-x}&e^{-x/2}&ie^{-x}\\
e^{-x}&-ie^{-x}&e^{-x/4}
\end{pmatrix},\, x\geq0,
$
and define the matrix-valued measure
$\sigma(B):=\int_B w(x)\,dx$
for every Borel subset \(B\) of \([0,\infty)\).
Then \(\sigma\) is a nonnegative Hermitian matrix-valued measure on
\([0,\infty)\). By \eqref{ss-j}, its moments are
\[
s_j
=
j!
\begin{pmatrix}
1&-i&1\\
i&2^{j+1}&i\\
1&-i&4^{j+1}
\end{pmatrix},
\qquad
j\geq0.
\]
It follows from the integral representations of
\(H_{1,2}\) and \(H_{2,2}\) that both matrices are Hermitian positive
definite.

For this degree-\(5\) example, the auxiliary matrices
\({\mathcal P}_{1,j}^{(0)}\) and
\({\mathcal Q}_{2,j}^{(0)}\), \(j=0,1,2\), are
\[
{\mathcal P}_{1,0}^{(0)}
=
\begin{pmatrix}
\dfrac73&i&-\dfrac13\\
-i&1&0\\
-\dfrac13&0&\dfrac13
\end{pmatrix},
\,
{\mathcal Q}_{2,0}^{(0)}
=
\begin{pmatrix}
1&-i&1\\
i&\dfrac43&i\\
1&-i&\dfrac85
\end{pmatrix},
\,
{\mathcal P}_{1,1}^{(0)}
=
\begin{pmatrix}
\dfrac{839}{255}&\dfrac{9i}{5}&-\dfrac{25}{51}\\
-\dfrac{9i}{5}&\dfrac95&0\\
-\dfrac{25}{51}&0&\dfrac{25}{51}
\end{pmatrix},
\]
\[
{\mathcal Q}_{2,1}^{(0)}
=
\begin{pmatrix}
\dfrac12&-\dfrac{i}{2}&\dfrac12\\
\dfrac{i}{2}&\dfrac{68}{111}&\dfrac{i}{2}\\
\dfrac12&-\dfrac{i}{2}&\dfrac{472}{655}
\end{pmatrix},
\,
{\mathcal P}_{1,2}^{(0)}
=
\begin{pmatrix}
\dfrac{12250171}{2955535}
&
\dfrac{1369i}{545}
&
-\dfrac{17161}{27115}
\\
-\dfrac{1369i}{545}
&
\dfrac{1369}{545}
&
0
\\
-\dfrac{17161}{27115}
&
0
&
\dfrac{17161}{27115}
\end{pmatrix},
\]
and
\[
{\mathcal Q}_{2,2}^{(0)}
=
\begin{pmatrix}
\dfrac13&-\dfrac{i}{3}&\dfrac13\\
\dfrac{i}{3}&\dfrac{80812}{206793}&\dfrac{i}{3}\\
\dfrac13&-\dfrac{i}{3}&\dfrac{5823608}{12891579}
\end{pmatrix}.
\]

Using \eqref{f5PP}, the corresponding HTM polynomial of degree \(5\)
is given by
\begin{equation}\label{fn5}
\mathbf f_5(z)
=
I_3z^5+A_1z^4+A_2z^3+A_3z^2+A_4z+A_5,
\end{equation}
where
\begin{align*}
A_1&=
\begin{pmatrix}
1&-i&1\\
i&2&i\\
1&-i&4
\end{pmatrix},
\quad
A_2=
\begin{pmatrix}
6&\dfrac{264i}{37}&-\dfrac{2464}{131}\\
0&\dfrac{486}{37}&0\\
0&0&\dfrac{3250}{131}
\end{pmatrix},
\\
A_3&=
\begin{pmatrix}
5&-5i&5\\
5i&\dfrac{560}{37}&5i\\
5&-5i&\dfrac{8440}{131}
\end{pmatrix},\quad
A_4=
\begin{pmatrix}
6&\dfrac{936i}{37}&-\dfrac{13152}{131}\\
0&\dfrac{1158}{37}&0\\
0&0&\dfrac{13938}{131}
\end{pmatrix},
\\
A_5&=
\begin{pmatrix}
2&-2i&2\\
2i&\dfrac{292}{37}&2i\\
2&-2i&\dfrac{9832}{131}
\end{pmatrix}.
\end{align*}
Thus, the moment sequence generated by the density \(w\) yields an
explicit HTM polynomial of degree \(5\). Since every HTM polynomial is
Hurwitz, it follows that \(\det\mathbf f_5\) is a Hurwitz polynomial.
\end{example}

\section{Determinants of the coefficients of HTM polynomials}\label{sec6}
In \cite[Remark~13]{ab2026}, it was conjectured that the
determinants of all coefficient matrices of HTM polynomials are
positive for every \(n\ge1\). In this section we prove that this
conjecture holds for HTM polynomials of degree at most \(3\), but
fails in general for degree \(4\). More precisely, we show that every
coefficient matrix \(A_j\) of the HTM polynomials
\(\mathbf f_1\), \(\mathbf f_2\), and \(\mathbf f_3\) satisfies
\[
\det A_j>0,
\]
whereas this property does not hold in general for
\(\mathbf f_4\).

To establish these results, we use the Hurwitz parameters
\((\cc_j)\) and \((\dd_j)\) introduced in
Definition~\ref{defhur}. By \cite[Theorem 7.12]{abH}), these parameters coincide with the
Dyukarev--Stieltjes parameters
\((M_j)\) and \((L_j)\), namely,
\begin{equation}\label{MMLL}
\cc_j=M_j,
\qquad
\dd_j=L_j,
\end{equation}
where
\begin{equation}\label{Mk}
M_k:=
\begin{cases}
s_0^{-1},
& k=0,\\[1ex]
v_k^{*}H_{1,k}^{-1}v_k-
v_{k-1}^{*}H_{1,k-1}^{-1}v_{k-1},
& k\ge1,
\end{cases}
\end{equation}
and
\begin{equation}\label{Lk}
L_k:=
\begin{cases}
s_0\,s_1^{-1}s_0,
& k=0,\\[1ex]
y_{[0,k]}^{*}H_{2,k}^{-1}y_{[0,k]}
-
y_{[0,k-1]}^{*}H_{2,k-1}^{-1}y_{[0,k-1]},
& k\ge1.
\end{cases}
\end{equation}
Here, recall that $H_{1,j}$, $H_{2,j}$ (see \eqref{61}) are positive definite and $v_k$, $y_{[0,k]}$ are defined by
\eqref{60a} and \eqref{65}.

Since the moment sequence \((s_j)_{j\ge0}\) is Stieltjes positive
definite, it follows from Equality~\eqref{MMLL} and
\cite[Remark~4.4]{abH} that each Hurwitz parameter
\(\cc_k\) and \(\dd_k\) is positive definite.

\begin{remark}\label{rem6.1}
Let \((\cc_k)\) and \((\dd_k)\) be the Hurwitz parameters defined by
\eqref{MMLL}, \eqref{Mk}, and \eqref{Lk}.
Expanding the right-hand sides of \eqref{cc13aa} and
\eqref{cc23aa} in powers of \(z\), and using
\eqref{hup02}, \eqref{hn}, and \eqref{gn}, we obtain the following
Hurwitz parametrizations of the HTM polynomials
\(\mathbf f_j\), \(j=1,\ldots,5\):
\begin{align}
\f_1(z)
={}&
I_qz+\cc_0^{-1},
\label{fcof1}
\\
\f_2(z)
={}&
I_qz^2+\cc_0^{-1}z+\dd_0^{-1}\cc_0^{-1},
\label{fcof2}
\\
\f_3(z)
={}&
I_qz^3+\cc_0^{-1}z^2
+(\cc_0+\cc_1)
\cc_1^{-1}\dd_0^{-1}\cc_0^{-1}z
+\cc_1^{-1}\dd_0^{-1}\cc_0^{-1},
\label{fcof3}
\\
\f_4(z)
={}&
I_qz^4+\cc_0^{-1}z^3
+
(\cc_1\dd_1+\cc_0\dd_0+\cc_0\dd_1)
\dd_1^{-1}\cc_1^{-1}\dd_0^{-1}\cc_0^{-1}z^2
\nonumber
\\
&+
(\dd_0+\dd_1)
\dd_1^{-1}\cc_1^{-1}\dd_0^{-1}\cc_0^{-1}z
+
\dd_1^{-1}\cc_1^{-1}\dd_0^{-1}\cc_0^{-1},
\label{fcof4}
\\
\f_5(z)
={}&
I_qz^5+\cc_0^{-1}z^4
\nonumber
\\
&+
(\cc_1\dd_1\cc_2
+\cc_0\dd_0\cc_2
+\cc_0\dd_0\cc_1
+\cc_0\dd_1\cc_2)
\cc_2^{-1}\dd_1^{-1}\cc_1^{-1}
\dd_0^{-1}\cc_0^{-1}z^3
\nonumber
\\
&+
(\dd_0\cc_2+\dd_0\cc_1+\dd_1\cc_2)
\cc_2^{-1}\dd_1^{-1}\cc_1^{-1}
\dd_0^{-1}\cc_0^{-1}z^2
\nonumber
\\
&+
(\cc_0+\cc_1+\cc_2)
\cc_2^{-1}\dd_1^{-1}\cc_1^{-1}
\dd_0^{-1}\cc_0^{-1}z
\nonumber
\\
&+
\cc_2^{-1}\dd_1^{-1}\cc_1^{-1}
\dd_0^{-1}\cc_0^{-1}.
\label{fcof5}
\end{align}
In \cite[Remark~9]{ab2020A}, the corresponding formulas in the scalar
case for polynomials of degrees \(1\) to \(5\) were expressed in terms
of the Dyukarev--Stieltjes parameters \(M_k\) and \(L_k\). Recall
from \eqref{MMLL} that these parameters coincide with
\(\cc_k\) and \(\dd_k\), respectively.
\end{remark}

Using the Hurwitz parametrizations in Remark~\ref{rem6.1}, together
with the positivity of the Hurwitz parameters \(\cc_j\) and
\(\dd_j\), we now prove the main result of this section.

\begin{theorem}\label{t-det}
Let \(\mathbf f_n\) be a \(q\times q\) HTM polynomial.
\begin{enumerate}\item[\textnormal{(a)}]
For every HTM polynomial of degree at most \(3\), all its coefficient
matrices have strictly positive determinants.
\item[\textnormal{(b)}]
The assertion in \textnormal{(a)} fails for degree \(4\). More
precisely, there exists an HTM polynomial of degree \(4\) having a
coefficient matrix with negative determinant.
\end{enumerate}
\end{theorem}

\begin{proof}
\textnormal{(a)}
Using the Hurwitz parametrization given in
Remark~\ref{rem6.1}, together with the positivity of the
Hurwitz parameters \(\cc_k\) and \(\dd_k\), we conclude that all
coefficient matrices of the HTM polynomials
\(\mathbf f_1\), \(\mathbf f_2\), and \(\mathbf f_3\)
have strictly positive determinants.

Indeed, for \(\mathbf f_3\),
\[
\det A_2
=
\det(\cc_0+\cc_1)\,
\det(\cc_1^{-1}\dd_0^{-1}\cc_0^{-1})>0,
\]
since \(\cc_0+\cc_1\) is positive definite, while
\(\cc_1^{-1}\), \(\dd_0^{-1}\), and
\(\cc_0^{-1}\) are positive definite and therefore have positive
determinants.

\medskip

\textnormal{(b)}
Consider the HTM polynomial $\mathbf f_4$ from
Example~\ref{exa1}. Its Hurwitz parameters are
\[
\cc_0=
\begin{pmatrix}
1&-1\\
-1&2
\end{pmatrix},
\,
\cc_1=
\begin{pmatrix}
1&1\\
1&2
\end{pmatrix},
\,
\dd_0=
\begin{pmatrix}
1&2\\
2&5
\end{pmatrix},
\,
\dd_1=
\begin{pmatrix}
1&-2\\
-2&5
\end{pmatrix}.
\]
Each of these matrices is positive definite and has determinant equal
to \(1\).

Using the representation \eqref{fcof4}, we obtain
\[
A_2
=
(\cc_1\dd_1+\cc_0\dd_0+\cc_0\dd_1)
\dd_1^{-1}\cc_1^{-1}\dd_0^{-1}\cc_0^{-1}
=
\begin{pmatrix}
-116&-43\\
391&145
\end{pmatrix},
\]
and therefore
\begin{equation}\label{A27}
\det(A_2)=-7.
\end{equation}
On the other hand,
$
\det(I_2)=1,\,
\det(A_1)=1,\,
\det(A_3)=20,\,
\det(A_4)=1.
$
Hence, the assertion of part \textnormal{(a)} does not extend to HTM
polynomials of degree \(4\).
\end{proof}

\section{Block Hurwitz matrices and Markov parameters}\label{sec7}
In this section, we establish explicit relationships between the block
Hurwitz matrix associated with an HTM polynomial, its Markov
parameters, and the corresponding block Hankel matrices. We first show
how the Markov parameters can be recovered from the coefficients of an
HTM polynomial. We then derive factorizations of the block Hurwitz
matrix and determinant identities that extend the corresponding
relations from the classical scalar Routh--Hurwitz theory.

The following lemma shows that the Markov parameters can be
recovered uniquely from the coefficients of an HTM polynomial.
Its scalar counterpart was established in~\cite[Lemma~3.1]{ab2018}.
\begin{lemma}\label{lem-sj}
Let \((A_j)_{j=0}^{n}\) be the coefficients of the HTM polynomial
\(\mathbf f_n\) defined by \eqref{hup01}, where \(A_0=I_q\).
Furthermore, let \(h_n\) and \(g_n\) be as in
\eqref{hn} and \eqref{gn}, respectively.

If \(n=2m\), then the Markov parameters
\((s_j)_{j=0}^{2m-1}\) are determined by
\begin{align}
\Ab_{2m}(s_0,s_1,\ldots,s_{2m-1})^*
&=
(A_1,A_3,\ldots,A_{2m-1},0_q,\ldots,0_q)^*.
\label{pq11}
\end{align}

If \(n=2m+1\), then the Markov parameters
\((s_j)_{j=0}^{2m}\) are determined by
\begin{align}
\Ab_{2m+1}(s_0,s_1,\ldots,s_{2m})^*
&=
(A_1,A_3,\ldots,A_{2m+1},0_q,\ldots,0_q)^*.
\label{pq22}
\end{align}

Here, for \(n\geq2\),
\[
\Ab_n
:=
\begin{pmatrix}
I_q & 0_q & \cdots & 0_q & 0_q\\
A_2^* & -I_q & \cdots & 0_q & 0_q\\
A_4^* & -A_2^* & \ddots & 0_q & 0_q\\
\vdots & \vdots & \ddots & \ddots & 0_q\\
A_{2(n-1)}^* & -A_{2(n-2)}^*
& \cdots & (-1)^nA_2^* & (-1)^{n+1}I_q
\end{pmatrix},
\]
where \(A_k:=0_q\) for \(k>n\).
\end{lemma}

\begin{proof}
We first consider the case \(n=2m\). By \eqref{hn} and \eqref{gn},
\[
h_{2m}(z)
=
\sum_{r=0}^{m}A_{2r}z^{m-r},
\qquad
g_{2m}(z)
=
\sum_{r=0}^{m-1}A_{2r+1}z^{m-1-r}.
\]
Substituting these expressions into \eqref{hup09}, multiplying by
\(h_{2m}(z)\), and comparing coefficients of equal powers of \(z\),
we obtain the triangular system
\[
\sum_{r=0}^{\ell}
(-1)^{\ell-r}s_{\ell-r}A_{2r}
=
\begin{cases}
A_{2\ell+1},
&0\leq\ell\leq m-1,\\
0_q,
&m\leq\ell\leq2m-1.
\end{cases}
\]
Taking adjoints and collecting these identities for
\(\ell=0,\ldots,2m-1\) yields precisely the block system
\eqref{pq11}.

The proof of \eqref{pq22} is analogous. Starting from
\eqref{hup10}, substituting the corresponding expressions for
\(h_{2m+1}\) and \(g_{2m+1}\), and comparing coefficients of equal
powers of \(z\), we obtain the associated triangular system.
Taking adjoints and collecting the resulting identities gives
\eqref{pq22}.

Finally, \(\Ab_n\) is block lower triangular with invertible diagonal
blocks \(I_q\) and \(-I_q\). Hence, \(\Ab_n\) is invertible, and the
systems \eqref{pq11} and \eqref{pq22} uniquely determine the stated
Markov parameters.
\end{proof}
Since the matrix \(\mathcal A_n\) is block lower triangular,
\begin{equation}\label{AA-11}
\det\mathcal A_{2m}=(-1)^{qm}, \quad \mbox{and} \quad \det\mathcal A_{2m+1}=(-1)^{qm}.
\end{equation}
Let \(n\geq2\), and let \(\mathbf f_n\) be the HTM polynomial
defined by \eqref{hup01}. 
Following the classical scalar theory, we define the associated
\(nq\times nq\) block Hurwitz matrix by
\begin{equation}\label{hffn}
H_{f}^{(n)}:=
\left(
\begin{array}{cccccc}
A_1 & A_3 & A_5 & \ldots & \ldots & 0\\
A_0 & A_2 & A_4 & \ldots & \ldots & 0\\
0 & A_1 & A_3 & \ldots & \ldots & 0\\
0 & A_0 & A_2 & \ldots & \ldots & 0\\
0 & 0 & A_1 & \ldots & \ldots & 0\\
0 & 0 & A_0 & \ldots & \ldots & 0\\
\ldots & \ldots & \ldots & \ldots & \ldots & \ldots\\
0 & 0 & 0 & \ldots & A_{n-1} & 0\\
0 & 0 & \ldots & \ldots & A_{n-2} & A_n
\end{array}
\right).
\end{equation}
Here \(A_r=0\) for \(r>n\). The matrix \(H_f^{(n)}\) is called the
\emph{block Hurwitz matrix} associated with the HTM polynomial
\(\mathbf f_n\).

We shall also use the following block matrices:
\begin{align}
{\mathcal L}_{2m}
:=
&
\left(
\begin{array}{cccccc}
s_0 & s_1 & s_2 & \ldots & \ldots & s_{2m-1}\\
I_q & 0_q & 0_q & \ldots & \ldots & 0_q\\
0_q & -s_0 & -s_1 & \ldots & \ldots & -s_{2m-2}\\
0_q & -I_q & 0_q & \ldots & \ldots & 0_q\\
\ldots & \ldots & \ldots & \ldots & \ldots & \ldots\\
0_q & 0_q & \ldots & (-1)^{m-1}s_0 & \ldots & (-1)^{m-1}s_m\\
0_q & 0_q & \ldots & (-1)^{m-1}I_q & \ldots & 0_q
\end{array}
\right),
\label{LL2ma}
\\
{\mathcal L}_{2m+1}
:=
&
\left(
\begin{array}{cccccc}
s_0 & s_1 & s_2 & \ldots & \ldots & s_{2m}\\
I_q & 0_q & 0_q & \ldots & \ldots & 0_q\\
0_q & -s_0 & -s_1 & \ldots & \ldots & -s_{2m-1}\\
0_q & -I_q & 0_q & \ldots & \ldots & 0_q\\
\ldots & \ldots & \ldots & \ldots & \ldots & \ldots\\
0_q & 0_q & \ldots & (-1)^m s_0 & \ldots & (-1)^m s_m
\end{array}
\right),
\label{LL2mp1a}
\end{align}
together with
\begin{equation}\label{S0mL2m}
S^{[0,m]}
:=
\left(
\begin{array}{ccccc}
0_q & 0_q & \ldots & 0_q & s_0\\
0_q & 0_q & \ldots & s_0 & s_1\\
\ldots & \ldots & \ldots & \ldots & \ldots\\
s_0 & s_1 & \ldots & s_{m-1} & s_m
\end{array}
\right),
\end{equation}
and
\begin{align}\label{Laa2m}
\Lambda_{2,m}
:=
\left(
\delta_{j,k}I_q
\right)_{
\tiny
\begin{array}{c}
j=0,\ldots,m\\
k=0,\ldots,m-1
\end{array}
}.
\end{align}
Here \(\delta_{j,k}\) denotes the Kronecker delta.

\begin{definition}\label{def:Tmatrices}
For \(m\geq1\), define
\begin{equation}\label{TT01}
{\mathcal T}_{2m}:=D_{2m}\widehat P_{2m},\qquad{\mathcal T}_{2m+1}:=D_{2m+1}\widehat P_{2m+1},
\end{equation}
where
$
\widehat P_{2m}:=
\left(p_{ij}I_q\right)_{i,j=1}^{2m},
\qquad
\widehat P_{2m+1}:=
\left(q_{ij}I_q\right)_{i,j=1}^{2m+1},
$
with
\[
p_{ij}:=
\begin{cases}
1,
& 1\leq i\leq m,\quad j=2i,\\[1ex]
1,
& m+1\leq i\leq2m,\quad j=4m-2i+1,\\[1ex]
0,
& \text{otherwise},
\end{cases}
\]
and
\[
q_{ij}:=
\begin{cases}
1,
& 1\leq i\leq m,\quad j=2i,\\[1ex]
1,
& m+1\leq i\leq2m+1,\quad j=4m-2i+3,\\[1ex]
0,
& \text{otherwise}.
\end{cases}
\]
Furthermore, let
$
D_{2m}
:=
\operatorname{diag}
\left(d_1^{(e)},\ldots,d_{2m}^{(e)}\right),
$
where
\[
d_i^{(e)}:=
\begin{cases}
(-1)^{i-1}I_q,
&1\leq i\leq m,\\[1ex]
(-1)^{2m-i}I_q,
&m+1\leq i\leq2m,
\end{cases}
\]
and let
$
D_{2m+1}
:=
\operatorname{diag}
\left(d_1^{(o)},\ldots,d_{2m+1}^{(o)}\right),
$
where
\[
d_i^{(o)}:=
\begin{cases}
(-1)^{i-1}I_q,
&1\leq i\leq m,\\[1ex]
(-1)^mI_q,
&i=m+1,\\[1ex]
(-1)^{2m+1-i}I_q,
&m+2\leq i\leq2m+1.
\end{cases}
\]
\end{definition}

The matrices \(\widehat P_{2m}\) and \(\widehat P_{2m+1}\) are the
block permutation matrices corresponding to the row reorderings in the
even and odd cases, respectively. The diagonal matrices
\(D_{2m}\) and \(D_{2m+1}\) eliminate the alternating sign pattern in
the reordered block rows. The following lemmas summarize the basic
algebraic properties of the resulting matrices \(\mathcal T_n\).

The matrices \(\mathcal T_n\) are unitary and have explicitly
computable determinants, as shown below.

\begin{lemma}\label{lem:Tunitary}
For every \(n\ge2\),
\[
\mathcal T_n^{-1}=\mathcal T_n^*.
\]
Equivalently,
\begin{equation}
\mathcal T_n^*\mathcal T_n
=
\mathcal T_n\mathcal T_n^*
=
I_{nq}.
\label{eq7.4C}
\end{equation}
\end{lemma}
\begin{proof}
Since \(D_n=D_n^*=D_n^{-1}\) and
\(\widehat P_n^{-1}=\widehat P_n^*\),
the result follows immediately from
\(\mathcal T_n=D_n\widehat P_n\).
\end{proof}

\begin{lemma}\label{lem:Tdet}
For every \(m\ge1\),
\begin{equation}\det\mathcal T_{2m}=\det\mathcal T_{2m+1}=(-1)^{qm}.\label{eq7.4}
\end{equation}
\end{lemma}

\begin{proof}
Since
$
\mathcal T_n=D_n\widehat P_n,
$
we have
$
\det\mathcal T_n=\det D_n\,\det\widehat P_n.
$

For \(n=2m\), the permutation underlying \(\widehat P_{2m}\) has
\(m^2\) inversions. Hence
\[
\det\widehat P_{2m}=(-1)^{qm^2}.
\]
Moreover,
$
\det D_{2m}=(-1)^{qm(m-1)}=1.
$
Therefore,
\[
\det\mathcal T_{2m}
=(-1)^{qm^2}
=(-1)^{qm}.
\]

For \(n=2m+1\), the permutation underlying
\(\widehat P_{2m+1}\) has \(m(m+1)\) inversions, and thus
\[
\det\widehat P_{2m+1}=(-1)^{qm(m+1)}=1.
\]
Also,
$
\det D_{2m+1}
=(-1)^{qm^2}
=(-1)^{qm}.
$
Consequently,
\[
\det\mathcal T_{2m+1}=(-1)^{qm}.
\]
\end{proof}

The following theorem establishes explicit relations among the block
Hurwitz matrix, the Markov parameters, and the associated block Hankel
matrices. In the scalar case, analogous identities were obtained in
\cite[Subsection~3.1 and Theorem~3.4]{ab2018}. Here we give their
matrix counterparts.
We also correct several typographical errors appearing in the proof of
\cite[Theorem~3.4]{ab2018}.
\begin{theorem}\label{th2.13}
Let \(n\geq2\), and let \(H_f^{(n)}\),
\(\mathcal L_{2m}\), and \(\mathcal L_{2m+1}\) be defined by
\eqref{hffn}, \eqref{LL2ma}, and \eqref{LL2mp1a}, respectively.

If \(n=2m\), then
\begin{align}
\bigl(H_f^{(2m)}\bigr)^*
&=
\Ab_{2m}\mathcal L_{2m}^*,
\label{hff2m}
\\
\det\!\bigl(\bigl(H_f^{(2m)}\bigr)^*\bigr)
&=
\det H_{2,m-1}.
\label{hff2mA}
\end{align}

If \(n=2m+1\), then
\begin{align}
\bigl(H_f^{(2m+1)}\bigr)^*
&=
\Ab_{2m+1}\mathcal L_{2m+1}^*,
\label{hff2m1}
\\
\det\!\bigl(\bigl(H_f^{(2m+1)}\bigr)^*\bigr)
&=
\det H_{1,m}.
\label{hff2m1A}
\end{align}
\end{theorem}

\begin{proof}
We first prove \eqref{hff2m}. Using \eqref{pq11},
\eqref{LL2ma}, and the convention \(A_r=0_q\) for \(r>n\), we see
that the first block column of
\(\Ab_{2m}\mathcal L_{2m}^*\) coincides with the first block column
of \(\bigl(H_f^{(2m)}\bigr)^*\).

The remaining block columns of \(\mathcal L_{2m}^*\) have the shifted
structure, with alternating signs and inserted identity blocks,
displayed in \eqref{LL2ma}. 

Multiplying these columns by the block lower triangular matrix
\(\mathcal A_{2m}\), and successively using the identities contained
in \eqref{pq11}, each block column is transformed into the
corresponding block column of
\((H_f^{(2m)})^*\).
We have
\[
\begin{pmatrix}I_q\\A_2^*\\A_4^*\\\vdots\end{pmatrix},
\qquad
\begin{pmatrix}0_q\\A_1^*\\A_3^*\\\vdots\end{pmatrix},
\qquad
\begin{pmatrix}0_q\\I_q\\A_2^*\\\vdots\end{pmatrix},
\qquad\ldots,
\]
where the terminal zero blocks arise from the above convention. These
are precisely the successive block columns of
\(\bigl(H_f^{(2m)}\bigr)^*\). Hence, \eqref{hff2m} holds.

Equality \eqref{hff2m1} is proved analogously, using
\eqref{pq22} and \eqref{LL2mp1a}.

To prove \eqref{hff2mA}, we use \eqref{eq7.4C}. We obtain
\begin{align*}
\bigl(H_f^{(2m)}\bigr)^*
&=
\Ab_{2m}\mathcal L_{2m}^* \\
&=
\Ab_{2m}
\left(\mathcal T_{2m}\mathcal L_{2m}\right)^*
\mathcal T_{2m}\\
&=
\Ab_{2m}
\begin{pmatrix}
I_{qm} & 0_{qm}\\
S^{[0,m-1]} & H_{2,m-1}
\end{pmatrix}^*
\mathcal T_{2m}.
\end{align*}
Taking determinants on both sides of the last equality, we obtain
\begin{align*}
\det\bigl(\bigl(H_f^{(2m)}\bigr)^*\bigr)
&=
\det\Ab_{2m}\,
\det\begin{pmatrix}
I_{qm} & 0_{qm}\\
S^{[0,m-1]} & H_{2,m-1}
\end{pmatrix}^*
\det\mathcal T_{2m}.
\end{align*}
Since \(H_{2,m-1}\) is Hermitian,
\[
\det
\begin{pmatrix}
I_{qm} & 0_{qm}\\
S^{[0,m-1]} & H_{2,m-1}
\end{pmatrix}^*
=
\det H_{2,m-1}.
\]
Moreover, by \eqref{AA-11} and \eqref{eq7.4},
\[
\det\Ab_{2m}\det\mathcal T_{2m}=1.
\]
Therefore, \eqref{hff2mA} follows.

Similarly, equality \eqref{hff2m1A} follows by using
\eqref{AA-11} and \eqref{eq7.4}, together with
\[
\det\Ab_{2m+1}\det\mathcal T_{2m+1}=1
\]
and
\[
\det
\begin{pmatrix}
I_{qm} & 0_{qm\times q(m+1)}\\
S^{[0,m]}\Lambda_{2,m} & H_{1,m}
\end{pmatrix}^*
=
\det H_{1,m},
\]
since \(H_{1,m}\) is Hermitian.
\end{proof}

The following remarks relate Theorem~\ref{th2.13} to the classical
scalar theory and illustrate some features of the matrix case.

\begin{remark}\label{rem3.11}
For \(n=2m\), identity \eqref{hff2m} appears in
\cite[Equality~(16)]{jaro} and in the second equality of
\cite[Formula~(36)]{jury}.

For \(n=2m+1\), the corresponding identity in
\cite[Equality~(5)]{jaro} 
is different from
 \eqref{hff2m1}. Likewise,
\cite[Formula~(36)]{jury} replaces \eqref{hff2m1} by a related formula
based on a different normalization.

The odd-degree formulation in
\cite{jaro,jury} is based on the asymptotic expansion
\[
\frac{g_{2m+1}(-z)}{h_{2m+1}(-z)}
=
\widetilde s_{-1}
-
\frac{\widetilde s_0}{z}
-
\frac{\widetilde s_1}{z^2}
-\cdots,
\]
which involves the additional Markov parameter
\(\widetilde s_{-1}\); see also
\cite[Equality~(1.5)]{ab2018}. The corresponding matrix-valued finite
continued fraction is given in
\cite[Equation~(2.6)]{abH}.
\end{remark}

\begin{remark}\label{rem7.8}
For the HTM polynomial \(\mathbf f_4\) in
Example~\ref{exa1}, the leading block Hurwitz determinants satisfy
$
\Delta_1=\Delta_2=\Delta_3=\Delta_4=1,
$
where
\[
\Delta_1=\det A_1,\quad
\Delta_2=
\det
\begin{pmatrix}
A_1&A_3\\
I_2&A_2
\end{pmatrix},\quad 
\Delta_3=
\det
\begin{pmatrix}
A_1&A_3&0_2\\
I_2&A_2&A_4\\
0_2&A_1&A_3
\end{pmatrix},
\]
and
\[
\Delta_4=
\det
\begin{pmatrix}
A_1&A_3&0_2&0_2\\
I_2&A_2&A_4&0_2\\
0_2&A_1&A_3&0_2\\
0_2&I_2&A_2&A_4
\end{pmatrix}.
\]
Hence all leading block Hurwitz determinants are positive.
On the other hand, Example~\ref{exa1} shows that
inequality~\eqref{A27} holds. Therefore, even within the class of HTM
polynomials, positivity of the leading block Hurwitz determinants does
not imply positivity of the determinants of all coefficient matrices.
\end{remark}

\end{document}